\documentclass[10pt,b5paper,twoside, headrule]{amsart}

\usepackage{amsfonts, amsmath, amssymb,latexsym}
\usepackage{epsfig}
\usepackage[curve]{xy}
\usepackage{algorithmic}
\usepackage{algorithm}
\usepackage{enumerate}
\usepackage{framed}
\usepackage{hyperref}
\usepackage{mathtools}

\usepackage{chngcntr}

\footskip=10mm\parskip 0.2cm\addtocounter{page}{0}
\newtheorem{thm}{Theorem}[section]

\newtheorem{lem}[thm]{Lemma}
\newtheorem{prop}[thm]{Proposition}

\theoremstyle{definition}
\newtheorem{defn}[thm]{Definition}

\theoremstyle{remark}
\newtheorem{rem}[thm]{Remark}
\numberwithin{equation}{section}

\title[RLWE/PLWE equivalence for totally real cyclotomic subextensions]{RLWE/PLWE equivalence for totally real cyclotomic subextensions via quasi-Vandermonde matrices}

\author{\sc Iv\'an Blanco-Chac\'on}
\address{Department of Mathematics, School of Science\\
Universidad de Alcal\'a de Henares}
\email{ivan.blancoc@uah.es}
\thanks{Partially supported by MTM2016-79400-P, CCG20/IA-057, CM/JIN/2019-031 and PID2019-104855RBI00/
AEI/10.13039/501100011033}

\begin{document}
\renewcommand\baselinestretch{1.2}
\renewcommand{\arraystretch}{1}
\def\base{\baselineskip}
\font\tenhtxt=eufm10 scaled \magstep0 \font\tenBbb=msbm10 scaled
\magstep0 \font\tenrm=cmr10 scaled \magstep0 \font\tenbf=cmb10
scaled \magstep0


\def\evenhead{{\protect\centerline{\textsl{\large{I. Blanco}}}\hfill}}

\def\oddhead{{\protect\centerline{\textsl{\large{On the non vanishing of the cyclotomic $p$-adic $L$-functions}}}\hfill}}

\pagestyle{myheadings} \markboth{\evenhead}{\oddhead}

\thispagestyle{empty}

\maketitle

\begin{abstract}We propose and justify a generalised approach to prove the polynomial reduction of the RLWE to the PLWE problem attached to the ring of integers of a monogenic number field. We prove such equivalence in the case of the maximal totally real subextension of the $4p$-th cyclotomic field, with $p$ arbitrary prime.
\end{abstract}

\bigskip
\section{Introduction}

The second round of the last NIST call confirms the lattice-based proposals as the strongest contenders (see https://www.safecrypto.eu/pqclounge/ for a description of the surviving candidates). Within the lattice category, RLWE/PLWE keeps the largest number of surviving proposals, other strong schemes being NTRU-Prime and FrodoKEM, a key encapsulation method based on LWE. These numbers, along with the ease-to-implement of most RLWE/PLWE-based primitives, relative small key sizes in comparison with code or multivariate-based schemes as well as encryption speed (specially in PLWE) and not the least, the fact of being a natural tool for fully homomorphic encryption, support the increasing interest in the topic from practical and theoretical points of view, and both inside and outside the Academia. 

A theoretical problem, which remains open in general, is the relation between RLWE, formulated in terms of rings of algebraic integers, and PLWE, in terms of rings of polynomials. The evaluation at an integral primitive element is an isomorphism between the underlying rings which may deform the error distributions and nothing prevents an exponential noise increase. This phenomenon has been studied in detail in \cite{RSW}, where a polynomial-time reduction was first established for an ad-hoc family of polynomials. 

However, for plenty of number theoretical reasons, it is the cyclotomic family the most interesting in cryptography, for which until now, such equivalence was an open question apart from the power-of-two and some particular cases \cite{DD}. In \cite{blanco}, we have proved a polynomial-time RLWE/PLWE-equivalence for cyclotomic number fields under the condition of fixing the number of primes dividing the conductor and a subexponential-time equivalence if we remove this condition. The author has been recently informed of \cite{italianos}, where a significant refinement of our main result in \cite{blanco} is obtained for conductors divisible by up to two primes.

In a nutshell, a good reason to pursue such equivalence results is that, roughly speaking, PLWE is proner to computer implementations while RLWE seems more suitable for security proofs: apart from \cite{stehle2}, which establishes the ideal SVP-to-PLWE reduction for the power of two case (which is equivalent to its RLWE version, since the evaluation map is in this case a scaled isometry), security reduction proofs are usually carried out using RLWE (cf. \cite{LPR} Theorem 4.1). On the other hand, a careful exploitation of the arithmetic of quotient polynomial rings allows for extremely efficient  algorithms. We must mention here the NTRU-Prime figures \cite{bernstein}: 28682 cycles on one core of an Intel Haswell CPU for polynomial multiplication in their recommended ring $\mathbb{F}_{4591}[x]/(x^{761}-x-1)$, at a postquantum security level of 128 bits. The speed of these calculations relies on Toom and Karatsuba's methods for polynomial multiplication over finite fields \cite{kara}, although some other high speed results had been attained using more standard techniques: for instance, NTT-based algorithms can achieve from 40000 down to 11722 cycles for NTRU classic \cite{ntruclassic}. One must compare these figures with the range of beyond 150000 cycles required for ECC-based multiplication (cf. for instance, Curve25519  in \cite{ecc}).

The present communication can be regarded as a continuation of \cite{blanco} to the setting of a non-trivial subextension of the cyclotomic field: its maximal totally real subfield. Beyond the interest in providing another example of polynomial-time RLWE/PLWE-equivalence for a family of fields not considered in previous research, we consider that the main value of this work is the novelty of our method: instead of considering the Vandermonde matrix in the Galois conjugates of the primitive element of the underlying number field, we replace it by an invertible matrix which depends on these roots in a very natural manner: a quasi-Vandermonde matrix whose entries are the Tchebycheff polynomials up to $m$, the degree of the extension, evaluated in the roots of the $(m+1)$-th Tchebycheff polynomial, which up to a scalar factor correspond with the Galois conjugates of the primitive element of our subextension.

This report is divided in two parts: the first is Section 2, where we set the notations and recall the facts we need from algebraic number theory (canonical embedding, rings of integers and monogeneicity) and the relevant definitions from the literature: the RLWE and PLWE problems, the notion of equivalence, the relation with the condition number and a summary of previous results in \cite{DD}, \cite{RSW} and \cite{blanco}. The second part is Section 3, where we expose our new approach in terms of quasi-Vandermonde matrices and prove our main result: the RLWE/PLWE-equivalence for maximal totally real subfields of the $4p$-th cyclotomic field (with $p$ odd prime). We end by providing some numerical examples and discussing some open problems, especially, to what extent we can exploit our approach for a more general setting (either totally real cyclotomic subfields with several primes in the conductor or to other abelian subextensions).

Along our work, by a lattice in $\mathbb{R}^n$ we mean an additive subgroup $\Lambda\subseteq \mathbb{R}^n$ which is isomorphic to $\mathbb{Z}^n$ (so, the condition of being full rank is implicit in our definition). The term $O(f(n))$ means, as usual, Laudau's big $O$ notation: a function $g(n)$ is $O(f(n))$ if there is a constant $C>0$, such that $|g(n)|\leq Cf(n)$ for big enough $n$.

\section{The general framework}

Let $K = \mathbb{Q}(\theta)$ be an algebraic number field of degree $n$ and let $f(x) \in\mathbb{Q}[x]$ be the minimal polynomial of $\theta$. Notice that $K$ is an $n$-dimensional $\mathbb{Q}$-vector space and the set $\{1,\theta,...,\theta^{n-1}\}$ is a $\mathbb{Q}$-basis. The evaluation-at-$\theta$ map is a field $\mathbb{Q}$-isomorphism $\mathbb{Q}[x]/(f(x))\cong K$ and will play a crucial role in this communication.

Recall that $K$ is endowed with $n$ field $\mathbb{Q}$-embeddings $\sigma_i: K \hookrightarrow \overline{\mathbb{Q}}$, with $1\leq i\leq n$ and $\overline{\mathbb{Q}}$ an algebraic closure of $\mathbb{Q}$, fixed from now on. The field $K$ is said to be Galois if it is the splitting field of $f$, what we will also assume here. 

Denote by $s_1$ the number of real embeddings, namely, those whose image is contained in $\mathbb{R}$. Denote by $s_2$ the number of complex non-real embeddings, so that $n=s_1+2s_2$. The canonical embedding $\sigma: K\to \mathbb{R}^{s_1}\times\mathbb{C}^{2s_2}$ is defined as 
$$
\sigma(x):=(\sigma_1(x),...,\sigma_n(x)).
$$ 
If $s_2=0$ ($s_1=0$), then $K$ is said to be totally real (imaginary).

As usual, $\mathcal{O}_K$ stands for the ring of algebraic integers of $K$. The field $K$ is said to be monogenic if $\mathcal{O}_K=\mathbb{Z}[\theta]$ for some $\theta\in K$, what we will also assume here.  In particular, $\mathcal{O}_K$ is a free $\mathbb{Z}$-module of rank $n$ \cite{stewart}, hence for each ideal $I\subseteq\mathcal{O}_K$ the image $\sigma(I)$ is a lattice (endowed with an extra ring structure inherited from $I$ via $\sigma$) in the space
$$
\Lambda_n:=\{(x_1,...,x_n)\in \mathbb{R}^{s_1}\times\mathbb{C}^{2s_2}: x_{s_1+i}=\overline{x}_{s_1+s_2+i}\mbox{ for }1\leq i\leq s_2\}.
$$
Such lattices are called ideal lattices in the RLWE literature. We will restrict ourselves here to the case $I=\mathcal{O}_K$. Notice that if $K$ is totally real, then $\Lambda_n=\mathbb{R}^n$.

We point out that multiplication and addition are preserved component-wise by the canonical embedding. This is not true for the coordinate embedding: for instance, for the ring $\mathbb{Z}[x]/(x^m+1)$, for $m=2^l$, multiplying by $x$  is equivalent to shifting the coordinates and negating the independent term.

\subsection{Cyclotomic fields and their subextensions} Let $n>1$ be an integer and denote by $\mathbb{Z}^*_n$ the group of multiplicative units in the ring $\mathbb{Z}_n$. As very well known, the set of primitive $n$-th roots of unity $\mu_n$ is a multiplicative group of order $m=\phi(n)$, where $\phi$ stands for the Euler's totient function. The $n$-th cyclotomic polynomial is
$$
\Phi_n(x)=\prod_{k\in\mathbb{Z}^*_n}(x-\zeta_k).
$$
This polynomial is irreducible and setting $\zeta=\zeta_k$ for any $k\in \mathbb{Z}^*_n$, the number field $K_n:=\mathbb{Q}(\zeta)$ is the splitting field of $\Phi_n(x)$, hence it is Galois of degree $m$. Moreover, it is monogenic \cite[Chap. 3]{stewart}.

If $n=p^r$ for $p$ prime and $r\geq 1$, taking $\zeta$ to be a multiplicative generator of $\mu_n$, one easily observes that every $\mathbb{Q}$-automorphism of $K_n$ is determined by the image of $\zeta$, for which there are $m=p^{r-1}(p-1)$ choices. With this observation plus the theorem of structure of finite abelian groups one has that
$$
\mathrm{Gal}(K_n/\mathbb{Q})\cong \mathbb{Z}_n^*.
$$
\begin{defn}A field extension $F/L$ is called abelian if $\mathrm{Gal}(F/L)$ is abelian. A number field $K$ is called abelian if the extension $K/\mathbb{Q}$ is so.
\end{defn}
Since $K_n$ is a Galois $\mathbb{Q}$-extension (moreover, abelian), the Galois correspondence is a bijection between the set of all subextensions of $K_n$ and the set of quotients of the group $\mathbb{Z}_n^*$. In particular, all the subsextensions of $K_n$ are abelian. Furthermore, the celebrated Kronecker-Weber theorem \cite{kw} states that every finite abelian $\mathbb{Q}$-extension $K$ is cyclotomic in the sense that $\mathrm{Gal}(K/\mathbb{Q})$ is a quotient of $\mathrm{Gal}(K_f/\mathbb{Q})\cong \mathbb{Z}_f^*$ for some $f\geq 1$, taken minimal, called the conductor of the extension. 

In this communication we will focus on the maximal totally real subextension of $K_n$, denoted $K_n^{+}$, namely, the largest subextension of $K_n$ such that its image by each Galois embedding is contained in $\mathbb{R}$.  This is a monogenic Galois extension of degree $m/2$, namely \cite[Chap. 1]{washington}, $\mathcal{O}_{K_n^{+}}=\mathbb{Z}[\psi_k]$ with $\psi_k:=\zeta_n^k+\zeta_n^{-k}=2cos\left(\frac{2k\pi}{n}\right)$, for each $k\in \mathbb{Z}_n^*/\{\pm 1\}$. We will denote by $\Phi_n^{+}(x)$ the minimal polynomial of $\psi_k$.

\subsection{The RLWE and PLWE problems}The next definitions apply for any number field, but we are interested here in $K=K_n$ and $K=K_n^{+}$. 

We will denote $\mathcal{O}=\mathbb{Z}[x]/(\Psi_n(x))$, where either $\Psi_n(x)=\Phi_n(x)$ or $\Phi_n^{+}(x)$. This ring is endowed with a lattice structure in $\Lambda_m$ (setting $m=\phi(n)$ in the first case and $m=\phi(m)/2$ in the second) via the coordinate embedding
$$
\begin{array}{ccl}\mathcal{O} & \longrightarrow & \Lambda_m\\
\displaystyle\sum_{i=0}^{m-1}a_i\overline{x}^i & \mapsto & (a_0,...,a_{m-1}),
\end{array}
$$
where $\overline{x}$ stands for the class of $x$ modulo the principal ideal generated by $\Psi_n(x)$. 
Notice that the image of $\mathcal{O}$ via the evaluation-at-$\zeta_n$ map when $K=K_n$ and at $\psi_n$ when $K=K_n^{+}$ yields a ring isomorphism $\mathcal{O}\cong\mathcal{O}_K$.

\begin{defn}[The RLWE/PLWE problem] Let $q=q(n)$ be a prime, with $q[x]\in\mathbb{R}[x]$, let $\chi$ be a discrete Gaussian distribution (cf., for instance, \cite[Section 2.2]{LPR}) with values in $\mathcal{O}_K/q\mathcal{O}_K$ (resp. in $\mathcal{O}/q\mathcal{O}$). The RLWE (resp. PLWE) problem for $\chi$ is defined as follows:

Given a \emph{secret} element $s\in \mathcal{O}_K/q\mathcal{O}_K$ (resp. $\mathcal{O}/q\mathcal{O}$) chosen uniformly at random, if an adversary for whom $s$ is hidden has access to arbitrarily many samples $\{(a_i,a_is+e_i)\}_{i\geq 1}$ of the RLWE (resp. PLWE) distribution, where for each $i\geq 1$, $a_i$ is uniformly chosen at random and $e_i$ is sampled from $\chi$, this adversary must recover $s$ with non-negligible advantage.
\end{defn}

As discussed in the introduction, both problems admit polynomial time quantum reductions from worst case SVP over ideal lattices, making them strong candidates for postquantum cryptography designs. The reduction for PLWE is given in \cite{stehle2} and the reduction for RLWE is given in \cite{LPR}.

\subsection{RLWE/PLWE equivalence. The condition number}

\begin{defn}Given a monogenic Galois number field $K=\mathbb{Q}(\theta)$ of degree $n\geq 2$, we say that RLWE and PLWE are equivalent for $K$ if every solution for the first can be turned in polynomial time into a solution for the second (and viceversa), incurring in a noise increase which is polynomial in $n$. In other words, the problems are equivalent if each one of them reduces to each other in polynomial time and with a polynomial noise increase.
\label{defneq}
\end{defn}
The topic is introduced and studied in \cite{DD} and in \cite{RSW}, where the following approach was followed:

Let $f(x)\in\mathbb{Z}[x]$ denote the minimal polynomial of $\theta$. Denote by $\theta_1:=\theta,\theta_2,...,\theta_n$ the Galois conjugates of $\theta$. As a lattice, $\mathbb{Z}[x]/(f(x))$ is endowed with the coordinate embedding while $\mathbb{Z}[\theta]$ is endowed with the canonical embedding, and the evaluation-at-$\theta$ isomorphism causes a distortion between both. Explicitly, the transformation between the embeddings caused by evaluation at $\theta$ is given by
\begin{equation}
\begin{array}{ccc}
V_{f}: \mathbb{Z}[x]/(f(x)) & \to & \sigma_1(\mathcal{O}_K)\times\cdots\times\sigma_n(\mathcal{O}_K)\\
\displaystyle\sum_{i=0}^{n-1}a_i\overline{x}^i & \mapsto & 

\left(\begin{array}{cccc}
1 & \theta_1 & \cdots & \theta_1^{n-1}\\ 
1 & \theta_2 & \cdots & \theta_2^{n-1}\\  
\vdots & \vdots & \ddots \vdots\\ 
1 & \theta_n & \cdots & \theta_n^{n-1}
\end{array}\right)
\left(\begin{array}{c}
a_0\\
a_1\\
\vdots\\
a_{n-1}
\end{array}\right).
\end{array}
\label{latticebij}
\end{equation}
Namely, the transformation $V_{f}$ is given by a Vandermonde matrix  acting on the coordinates. 

For a matrix $A\in\mathrm{M}_n(\mathbb{C})$, we will denote from now on by $||A||$ the Frobenius norm of $A$, namely, $||A||:=\sqrt{Tr(AA^*)}$, where $Tr$ stands for the trace map and $A^*$ is the conjugated-transpose of $A$. Notice that if $A=(a_{i,j})_{i,j=1}^n$, then 
\begin{equation}
||A||=\sqrt{\sum_{i,j=1}^na_{i,j}^2}.
\end{equation}
In particular, if $A_i$ is a principal submatrix of $A$ then $||A_i||\leq ||A||$.

As usual, $||A||_{\infty}$ will denote the infinity norm, namely the largest entry of $A$ in absolute value. 

As discussed in \cite[Section 4.2]{RSW}, the noise growth caused by $V_{f}$ will remain \emph{controlled} whenever $||V_{f}||$ and $||V_{f}^{-1}||$ remain so, thus a meaningful measure of how both quantities are controlled is given by the condition number of $V_{f}$:

\begin{defn}The condition number of an invertible matrix $A\in\mathrm{M}_n(\mathbb{C})$ is defined as Cond$(A):=||A|||A^{-1}||$.
\end{defn}
Here are some properties of the Frobenius norm and the condition number which we will use in next section:
\begin{prop}Let $A,B\in GL_n(\mathbb{C})$ be any invertible matrices. We have:
\begin{itemize}
\item The condition number is invariant by scalar multiplication, namely, for each $\lambda\in\mathbb{C}^*$ it is $||A||=|\lambda|||A||$ and $\mathrm{Cond}(\lambda A)=\mathrm{Cond}(A)$. 
\item The condition number satisfies $\mathrm{Cond}(A)=\mathrm{Cond}(A^{-1})$.
\item The condition number is submultiplicative, namely: 
$$
\mathrm{Cond}(AB)\leq \mathrm{Cond}(A)\mathrm{Cond}(B).$$
\end{itemize}
\label{propcond}
\end{prop}
\begin{proof}The first and second claims are straightforward, the third reduces to the well known inequality $||AB||\leq ||A||||B||$.
\end{proof}
Hence the problem of the equivalence is reduced to show that $\mathrm{Cond}(V_{f})=O(n^r)$ for $r$ independent of $n$.   The easiest case is that of $f(x)=\Phi_n(x)$, with $n=2^l$, discussed in \cite{stehle2}. It is easy to show that $V_{\Phi_n}$ is an scaled isometry, with scaling factor $m=\phi(n)$. The general cyclotomic case is dealt with in \cite{blanco} and we recall the main results and ideas in the next subsection for conveninece of the reader.

In \cite[Thm 4.7]{RSW},  such a result is proved for a family of polynomials of the form $x^n+xp(x)-r$ with $deg(p(x))<n/2$, where $r=r(n)$, $r(x)\in\mathbb{R}[x]$. 

The key difficulty in generalising these ideas to wider classes of number fields is that Vandermonde matrices tend to be very badly conditioned. The case of complex nodes shows up in our work \cite{blanco} for the cyclotomic case and required non-trivial bounds based on ideas of analytic number theory going back to Erd\"os and Bateman. 

The case of the totally real subextension of cyclotomic number fields is even harder, at least with the Vandermonde approach from \cite{RSW}. Moreover, as we show next, it is condemned to failure:

\begin{thm}Given a collection $\textbf{s}$ of $n$ real nodes, the attached Vandermonde matrix $V_{\textbf{s}}$ is exponentially conditioned (at least) in these situations:
\begin{itemize}
\item When all the nodes are positive, one has $\mathrm{Cond}(V_{\textbf{s}}) > 2^{n-1}.$
\item When the nodes are symmetrically located with respect to the origin, one has $\mathrm{Cond}(V_{\textbf{s}}) > 2^{n/2}.$
\end{itemize}
\label{thmexpgautschi}
\end{thm}
\begin{proof}The first case is Theorem 2.1 of  \cite{gautschi} and the second is Theorem 3.1 of loc. cit. Both results refer to the condition number attached to the infinity norm. However, for any invertible matrix $A\in\mathrm{M}_n(\mathbb{R})$, it is straightforward to check that $||A||\geq ||A||_{\infty}$.
\end{proof}

Unlike the cyclotomic case, the Vandermonde matrix $V_{\Phi_{n}^{+}}$ which transforms the lattice $\mathcal{O}$ in the lattice $\sigma(\mathcal{O}_{K_n^{+}})$ has real nodes and these are symmetrically localed with respect to the origin, at least when $4\mid n$:
\begin{prop}Let $n\geq 2$ and assume $4\mid n$. The nodes $\psi_k$ corresponding to the number field $K_n^{+}$ are symmetrically located with respect to the origin.
\label{expnodes}
\end{prop}
\begin{proof}
The ratios $\frac{2k}{n}$ with $k\in(\mathbb{Z}/n\mathbb{Z})^*$ and $k\neq n/4$ are distributed in two classes: those with $k\in (1, n/4)$ and those with $k\in (n/4,n/2)$. Given $k \in (n/4,n/2)$ coprime with $n$, we have that $n/2-k\in (1, n/4)$, $n/2-k$ is also coprime with $n$, and $\cos\left(\frac{2k\pi}{n}\right)=-\cos\left(\frac{2(n/2-k)\pi}{n}\right)$.
\end{proof}

Hence, invoking Thm. \ref{thmexpgautschi}, we have $\mathrm{Cond}(V_{\Phi_{n}^{+}})>2^{m/2}$ and there is no hope of polynomial RLWE/PLWE-equivalence at least via the approach in \cite{RSW}, based in $V_{\Phi_n^{+}}$. Our main result in Section 3 addresses how to replace $V_{\Phi_n^{+}}$ by a so called \emph{quasi-Vandermonde} matrix with a polynomially bounded condition number.
\begin{rem}[On the notion of equivalence] As suggested by one of the referees, we mention that the transformation \eqref{latticebij} defined by the Vandermonde matrix $V_f$ is not only a lattice isomorphism, but also a ring isomorphism and if we look at Def. \ref{defneq}, this is far more than what is necessary for PLWE and RLWE to be equivalent. Being equivalent just means that there exist polynomial time reduction algorithms which take PLWE-samples to RLWE-samples and viceversa incurring in a noise increase which is polynomial in the degree of the base field, that is all. 

For instance, in \cite{RSW}, the reduction in Theorem 4.2 is given by the map
$$
\begin{array}{ccc}
\Psi: \mathcal{O}_K/q\mathcal{O}_K\times K_{\mathbb{R}}/\mathcal{O}_K & \to & \mathcal{O}/q\mathcal{O}\times K_{\mathbb{R}}/\mathcal{O}\\
(a,b) & \mapsto (ta,t^2b),
\end{array}
$$
which is not a ring homomorphism.

The equivalence that we will establish between RLWE and PLWE for the totally real cyclotomic subextension will be given by a lattice isomorphism between $\mathbb{Z}[x]/(\Phi_{4p}^{+}(x))$ (with the coordinate embedding) and $\sigma\left(\mathcal{O}_{K_{4p}^{+}}\right)$ (with the canonical embedding) which is not a ring isomorphism.
\end{rem}
\begin{rem}[On the coordinate and canonical embeddings] In the literature, for a monogenic field $K$, the coordinate embedding in $\mathbb{Z}[x]/(f(x))$ depends on the power basis $\{1,x,...,x^{m-1}\}$ and the coordinate embedding depends on the $\mathbb{Z}$-basis of $\mathcal{O}_K$ given by $\{1,\theta,...,\theta^{m-1}\}$, with $\theta$ an integral primitive element of $K$. Since bases of integers (and bases of polynomial quotients) are not unique, in fairness, the notion of \emph{coordinate embedding} depends of a basis and analogously for the \emph{canonical embedding}. Our equivalence result will be between PLWE for $\mathbb{Z}[x]/(f(x))$ with the coordinate embedding with respect to the usual power basis and RLWE for $\mathcal{O}_{K_{4p}^{+}}$ with the canonical embedding with respect to a $\mathbb{Z}$-basis of $\mathcal{O}_{K_{4p}^{+}}$ which is not the usual power basis. This does not mean that we are changing the RLWE problem, all we are doing is to change the $\mathbb{Z}$-basis of $\mathcal{O}_{K_{4p}^{+}}$, or equivalently, the $\mathbb{Z}$-basis of the target lattice $\sigma(\mathcal{O}_{K_{4p}^{+}})$.
\end{rem}
\subsection{The cyclotomic case}

Notations for $\Phi_n(x)$, $\zeta$, $K_n$ and $\mathcal{O}_{K_n}$ are as in Section 2.1. Denote $m=\phi(n)$ as usual.

\begin{defn}For $n\geq 2$, let $A(n)$ denote the maximum coefficient of $\Phi_n(x)$ in absolute value. If $n=p_1^{r_1}...p_s^{r_s}$, denote $rad(n)=p_1...p_s$.
\label{defncyclo}
\end{defn}

Since the $2$-power case is a scaled isometry, we assume $rad(n)\neq 2$. Our main result was as follows:

\begin{thm}\cite[Thm. 3.10]{blanco} For $k\geq 1$, assume $rad(n)=p_1...p_k$. Then:
$$
Cond(V_{\Phi_n})\leq 2rad(n)n^{2^k+k+2}A(n).
$$
Consequently,  if $k$ is fixed, then $Cond(V_{\Phi_n})$ is polynomial in $n$. 
\label{main2}
\end{thm}

The keys of the proof are: 1) as in \cite{RSW}, we start with an expression of the entries in $V_{\Phi_n}^{-1}$ as quotients of symmetric polynomials in the $n$-th primitive roots, 2) a bound for $A(n)$ due to Bateman \cite{bateman2} which is polynomial in $m$ once $k$ is fixed; some surgery on this bound allows to control the numerators, and 3) the observation that $A(n)=A(rad(n))$, which simplifies the treatment of the denominators. When $k\leq 3$, we can refine our bound as follows:

\begin{thm}\cite[Thms. 4.1-4.3]{blanco} Let $n\geq 1$ and let $rad(n)$ be divisible by at most $k\leq 3$ primes. Then:
$$
Cond(V_{\Phi_n})\leq 4\phi(rad(n))m^k.
$$
\end{thm}

We have recently been aware of \cite{italianos}, where a closed formula is given for the condition number if $n=2^kp^l$, with $k,l\geq 0$. This is Theorem 1.2 of loc. cit. Although all what matters for the equivalence is to grant a polynomial bound, having a sharper one as in \cite{italianos} might be useful for instance, for a hyptohetical cryptanalysis of RLWE via PLWE as in \cite{elias}, not for $\alpha=1$ (as this is never a root of $\Phi_n(x)$), but for $\alpha$ of small/medium order in $\mathbb{Z}_n^*$.

\section{Generalized equivalence for the maximal totally real extension}

Let $n\geq 1$ be fixed in this section unless stated otherwise and for $1\leq k\leq n$ coprime to $n$, recall that $\psi_k$ is a primitive element of $K_n^{+}$. Denote by $\Phi_n^{+}(x)$, as in Section 2.1,  the minimal polynomial of all the $\psi_k's$, which are Galois-conjugated of each other, and set $\mathcal{O}:=\mathbb{Z}[x]/(\Phi_n^{+}(x))$. To ease notation, we set in this section $m=\phi(n)/2$, the degree of $K_n^{+}$. 

As we have proved via Prop. \ref{expnodes} and Thm. \ref{thmexpgautschi}, we cannot use the Vandermonde matrix $V_{\Phi_n^{+}}$ to establish the RLWE/PLWE equivalence for $K_n^{+}$. Next, we replace it by another invertible matrix with entries in $\mathcal{O}_{K_n^{+}}$ with a condition number which is polynomially bounded in $m$.

The starting observation is that in absence of errors, the approach in \cite{RSW} allows to pass from a RLWE-sample to a PLWE-sample by solving the linear system attached to the matrix $V_{K_n^{+}}$ via Gaussian elimination, which takes $O(m^3)$ operations. When we add the errors, all the business is spoiled by the exponential amplification of them, caused by the condition number of $V_{K_n^{+}}$: if the variance increases beyond a limit, a valid PLWE sample may be turned into a non-valid RLWE sample and viceversa (i.e. decryption may become unfeasible).

Our idea is to replace the lattice isomorphism between both sample spaces: instead of using $V_{K_n^{+}}$ we transfer the samples by multiplication with a quasi-Vandermonde matrix 
$$QV_{K_n^{+},\{p_i(x)\}_{i=0}^{m-1}}=\left(
\begin{array}{cccc}
p_0(\psi_1) & p_1(\psi_1) & ... & p_{m-1}(\psi_1)\\
p_0(\psi_2) & p_1(\psi_2) & ... & p_{m-1}(\psi_2)\\
\vdots & \vdots & \ddots & \vdots\\
p_0(\psi_m) & p_1(\psi_m) & ... & p_{m-1}(\psi_m)\\
\end{array}
\right),$$
where $p_i(x)\in\mathbb{Z}[x]$ has degree $i$. The requirement that $p_i(x)\in\mathbb{Z}[x]$ for each $0 \leq i\leq m-1$ ensures that $QV_{K_n^{+},\{p_i(x)\}_{i=0}^{m-1}}$ defines a monomorphism of lattices from $\mathcal{O}$ to $\sigma(\mathcal{O}_{K_n^{+}})$:
\begin{prop}For every choice $\{p_i(x)\}_{i=0}^{m-1}\subseteq\mathbb{Z}[x]$ with $\mathrm{deg}(p_i(x))=i$, the matrix $QV_{K_n^{+},\{p_i(x)\}_{i=0}^{m-1}}$ is invertible.
\label{propinv}
\end{prop}
\begin{proof}First, since $K_n^{+}$ is Galois and monogenic, setting $\psi=\psi_1$, the set $\{\psi^j\}_{j=0}^{m-1}$ is an integral basis of $K_n^{+}$ and so is $\{\psi^{k_j}\}_{j=0}^{m-1}$ for each complete reduced ordered system of reminders $k_j\in\mathbb{Z}_n^*/\{\pm 1\}\cong\mathrm{Gal}(K_n^{+}/\mathbb{Q})$. By letting the Galois group act on $\psi$, we have that $\{\psi_k^{k_j}\}_{j=0}^{m-1}$ is also a basis for each $1\leq k\leq n$ coprime to $n$.  Secondly, $m$ polynomials of different degree are always linearly independent, hence $\{p_i(\psi_k)\}_{i=0}^{m-1}$ is also a basis for each $k$. 

To see that $QV:=QV_{K_n^{+},\{p_i(x)\}_{i=0}^{m-1}}$ is invertible is enough to check that there are no non-trivial solutions $\textbf{x}\in K_{\Phi_n^{+}}^m$ of the system $QV\textbf{x}=\textbf{0}$, which is immediate from the two previous observations.
\end{proof}
Moreover, if the polynomials are chosen in such a way that $\{p_i(\psi_j)\}_{i=0}^{m-1}$ is also a $\mathbb{Z}$-basis of $\mathcal{O}_{K_n^{+}}$ for some, and hence all $j$, then the transformation is actually a lattice isomorphism. In particular, we have:
\begin{prop}Notations as before, if for each $i\in \{0,m-1\}$ the polynomial $p_i(x)\in \mathbb{Z}[x]$ is monic then the transformation defined by $QV$ is a lattice isomorphism.
\end{prop}
\begin{proof}It is enough to check that the transformation is surjective, and this is immediate if we prove that $\{p_i(\psi_j)\}_{i=0}^{m-1}$ is a $\mathbb{Z}$-basis of $\mathcal{O}_{K_n^{+}}$ for some and hence all $j$. But this follows from the fact that the polynomials are monic, since given $\alpha\in \mathcal{O}_{K_n^{+}}$, setting for instance $j=1$, let us write $\alpha=b_0+b_1\psi+...+b_{m-1}\psi^{m-1}$. Writing $p_i(x)=\sum_{j=0}^ir_{i,j}x^j$, we must find $(a_0,...,a_{m-1})\in\mathbb{Z}^m$ such that
$$
\left(
\begin{array}{c}
b_0\\
b_1\\
\vdots\\
b_{m-1}
\end{array}
\right)=
\left(
\begin{array}{ccccc}
r_{0,0} & r_{1,0} & r_{2,0} & \cdots & r_{m-1,0}\\
0 & r_{1,1} & r_{2,1} & \cdots & r_{m-1,1}\\
0 & 0 & r_{2,2} & \cdots & r_{m-1,2}\\
\vdots   & \vdots   & \vdots   & \ddots & \vdots\\
0 & 0 & 0 & \cdots & r_{m-1,m-1}\\
\end{array}
\right)\left(
\begin{array}{c}
a_0\\
a_1\\
\vdots\\
a_{m-1}
\end{array}
\right).
$$
But since by hypothesis $r_{i,i}=1$ for each $i$, the matrix is unimodular, the system has integer solutions and the result follows.
\end{proof}
Besides taking the polynomials with integer coefficients, different degrees and in such a way that $\{p_i(\psi_j)\}_{i=0}^{m-1}$ is a $\mathbb{Z}$-basis, there is clearly another constraint to meet:  we need to take them so that $\mathrm{Cond}(QV_{K_n^{+}, \{p_i(x)\}_{i=0}^{m-1} })$ is polynomial in $m$, whenever this is possible. In our case we can attain these requirements, as we will see next, by using the following family of polynomials:

\begin{defn}The family of Tchebycheff polynomials of the first kind is defined by any of the following equivalent properties:
\begin{itemize}
\item[a)] $T_i(x)=\cos(i\arccos(x))$ for $i\geq 1$.
\item[b)] $T_0(x)=1, T_1(x)=x$ and $T_i(x)=2xT_{i-1}(x)-T_{i-2}(x)$ for $i\geq 2$.
\end{itemize}
\label{tcheby}
\end{defn}

The essential reason why this approach solves our problem is the following result:
\begin{prop}For $N\geq 1$, let $x_k^{(N)}:=\cos\left( \frac{2k-1}{2N}\pi \right)$, with $1\leq k \leq N$. Denote $V_N=(T_i(x_k^{(N)})_{i,k-1=0}^{N-1}$. Then, $\mathrm{Cond}(V_N)\leq N(N+1)$. 
\label{kr}
\end{prop}
\begin{proof}Define $P_0(x)=\frac{1}{\sqrt{\pi}}T_0(x)$ and $P_j(x)=\sqrt{\frac{2}{\pi}}T_j(x)$ for $j\geq 1$. Setting $W_N:=(P_i(x_k^{(N)}))_{i,k-1=0}^{N-1}$, from \cite[Cor. 1]{kuian} we obtain that $\mathrm{Cond}(W_N)=N$. 

Hence, setting $T_0^{*}(x)=\frac{1}{\sqrt{2}}T_0(x)$ and $T_j^{*}(x)=T_j(x)$ for $j\geq 1$, for the matrix $W_N^*:=(T^{*}_i(x_k^{(N)}))_{i,k-1=0}^{N-1}$, due to Prop. \ref{propcond}, it is also $\mathrm{Cond}(W_N^{*})=N$.

Finally, $V_N=W^{*}_ND$, where $D$ is the diagonal matrix having $\sqrt{2}$ at position $(1,1)$ and $1$ in the rest of entries. Now, $||D||=\sqrt{N+1}$ and $||D^{-1}||=\sqrt{N-1/2}$ and using again Prop. \ref{propcond} the result follows.
\end{proof}

To use this bound, another step to fix is that the nodes $\{\psi_{2k-1}\}_{k=1}^N=\{2x_k^{(N)}\}_{k=1}^N$, i.e., both collections of nodes differ by a factor $2$, and unlike the Vandermonde case, in principle, we cannot pull the scalar out as a common factor of all the entries and invoke the scalar invariance of the condition number. To solve this, let us define $Q_i(x):=T_i(\frac{1}{2}x)$. We can easily prove by induction the following:
\begin{lem}For $n\geq 1$, we can write $Q_n(x)=\frac{1}{2}R_n(x)$, where $R_n(x)\in\mathbb{Z}[x]$. Furthermore, $R_n(x)$ is monic and its independent term $r_{n,i}$ satisfies
$$
r_{n,0}=\left\{
\begin{array}{l}
0\mbox{  if }n>0\mbox{ is even,}\\
2\mbox{  if }n=0\mbox{ or }n\equiv 1\pmod{4},\\
-2\mbox{ otherwise.}
\end{array}
\right.
$$
\label{lemaint}
\end{lem}
Hence, we can write
$$
V_N=(T_i(x_k^{(N)})_{i,k-1=0}^{N-1}=(Q_i(2x_k^{(N)})_{i,k-1=0}^{N-1}=\frac{1}{2}(R_i(\psi_{2k-1})_{i,k-1=0}^{N-1},
$$
and we have, by Prop. \ref{propcond}:
\begin{equation}
\mathrm{Cond}((R_i(\psi_{2k-1})_{i,k-1=0}^{N-1})=\mathrm{Cond}(V_N)\leq N(N+1).
\label{surgery}
\end{equation} 
We are now in position to state and prove our main result.
\subsection{The equivalence for $K_{4p}^{+}$}
In this case $N=p$ and our extension $K_{4p}^{+}/\mathbb{Q}$ has degree $p-1$, and for $k\neq \frac{p+1}{2}$, the nodes $\{x_k^{(p)}\}_{k=1}^p$ are the Galois conjugates of $x_1^{(p)}$. Denote $Q_{4p}:=\left(R_i(\psi_{2k-1})\right)_{i,k-1=0}^{p-1}$. Using Prop \ref{kr} and Eq. \eqref{surgery}, we have:
\begin{equation}
\mathrm{Cond}(Q_{4p})\leq  p(p+1).
\label{condpartial1}
\end{equation}
Our last problem is that since we are seeking for a lattice-isomorphism between $\mathcal{O}$ and $\sigma(\mathcal{O}_{K_{4p}^{+}})$, we are only interested in the rows of $Q_{4p}$ corresponding to the values of $k$ such that $2k-1$ is coprime to $p$. Since $1\leq k\leq p$, we must exclude precisely the row $k=\frac{p+1}{2}$.

Now, by the very definition we have $R_i(\psi_{\frac{p+1}{2}})=2\cos\left(\frac{i\pi}{2}\right)$ for $0\leq i\leq p-1$ and hence, the $\frac{p+1}{2}$-th row of $Q_{4p}$ has entries in $\{0,\pm 2\}$. 

It is easy to see that permutation of two rows does not affect the Frobenius norm, hence neither the condition number. Thus, we still denote by $Q_{4p}$ the result of permuting the first and $\frac{p+1}{2}$-th rows:
$$
Q_{4p}=\left(
\begin{array}{ccccc}
2 & 0 & -2 & \cdots & \epsilon\\
R_0(\psi_1) & R_1(\psi_1) & R_2(\psi_1) & \cdots & R_{p-1}(\psi_1)\\
R_0(\psi_2) & R_1(\psi_2) & R_2(\psi_2) & \cdots & R_{p-1}(\psi_2)\\
\vdots & \vdots & \vdots & \ddots & \vdots\\
R_0(\psi_{p-1}) & R_1(\psi_{p-1}) & R_2(\psi_{p-1}) & \cdots & R_{p-1}(\psi_{p-1})\\
\end{array}
\right),
$$
where $\epsilon=0$ if $p\equiv 1\pmod{4}$ and $\pm 2$ otherwise. Now, we can prove:
\begin{prop}We can write
$$
M_{4p}:=FQ_{4p}C=\left(
\begin{array}{cc}
2 & \textbf{0}^t\\
\textbf{0} & N_{4p}
\end{array}
\right),
$$
where
$$
F=\left(
\begin{array}{rcccc}
1 & \phantom{a} & 0  & \cdots & 0\\
-1 & \phantom{a} & 1 & \cdots & 0\\
\vdots & \phantom{a} & \vdots & \ddots & \vdots\\
-1 & \phantom{a} & 0  & \cdots & 1
\end{array}
\right); \mbox{ }
C=\left(
\begin{array}{cc}
1 & \textbf{r} \\
\textbf{0} & I 
\end{array}
\right),
$$
with $\textbf{0}\in \mathbb{R}^{p-1}$ the zero vector, $\textbf{r}=\left(
\begin{array}{ccccccc}
0 & 1 & 0 & -1 & 0 & 1 & \cdots
\end{array}
\right)\in\mathbb{R}^{p-1}$ and
$$
N_{4p}=\left(
\begin{array}{cccc}
R_1^*(\psi_1) & R_2^*(\psi_1)  & \cdots & R_{p-1}^*(\psi_1)\\
R_1^*(\psi_2) & R_2^*(\psi_2) &  \cdots & R_{p-1}^*(\psi_2)\\
\vdots & \vdots &  \ddots & \vdots\\
R_1^*(\psi_{p-1}) & R_2^*(\psi_{p-1})  & \cdots & R_{p-1}^*(\psi_{p-1})\\
\end{array}
\right).
$$ 
Moreover, for each $i\in\{1,...,p-1\}$, the polynomial $R_i^*(x)\in\mathbb{Z}[x]$ is monic, has no independent term and $\mathrm{deg}(R_i^*(x))=i$. In particular the matrix $N_{4p}$ is invertible and
$$
\mathrm{Cond}(N_{4p})\leq p(p+1)(2p-1)^2.
$$
\label{reduction1}
\end{prop}
\begin{proof}First, we multiply by the column operation matrix C to obtain:
$$
N_{4p}^{(1)}:=Q_{4p}C=\left(
\begin{array}{ccccc}
2 & 0 & 0 & \cdots & 0\\
R_0^*(\psi_1) & R_1^*(\psi_1) & R_2^*(\psi_1) & \cdots & R_{p-1}^*(\psi_1)\\
R_0^*(\psi_2) & R_1^*(\psi_2) & R_2^*(\psi_2) & \cdots & R_{p-1}^*(\psi_2)\\
\vdots & \vdots & \vdots & \ddots & \vdots\\
R_0^*(\psi_{p-1}) & R_1^*(\psi_{p-1}) & R_2^*(\psi_{p-1}) & \cdots & R_{p-1}^*(\psi_{p-1})\\
\end{array}
\right),
$$
where $R_0^*(x)=R_0(x)$, and
\begin{equation}
R_i^*(x)=\left\{
\begin{array}{l}
R_i(x)\mbox{ if } i\mbox{ is even}\\
R_i(x)\pm R_0(x)\mbox{ if } i\mbox{ is odd.}
\end{array}
\right.
\end{equation}
We observe, first, that the polynomials $R_i^*(x)$ have no independent term and are monic due to Lemma \ref{lemaint}, and second, that since $C$ has $1's$ over the diagonal and $\pm 1$ over (some of) the odd positions of the first row and the rest of terms are zero, it follows that 
\begin{equation}
\mathrm{Cond}(C)\leq 2p-1.
\label{eqcond1}
\end{equation}
Next, we multiply on the left by $F$ to obtain the required decomposition. Similarly as with $C$, we obtain
\begin{equation}
\mathrm{Cond}(F)\leq 2p-1.
\label{eqcond2}
\end{equation}
Now, we have
\begin{equation}
||N_{4p}||\leq||M_{4p}||\leq||F||||Q_{4p}||||C||,
\label{eqcond3}
\end{equation}
and since $M_{4p}^{-1}=C^{-1}Q_{4p}^{-1}F^{-1}=\left(
\begin{array}{cc}
1/2 & \textbf{0}^t\\
\textbf{0} & N_{4p}^{-1}
\end{array}
\right)$, we also have
\begin{equation}
||N_{4p}^{-1}||\leq||M_{4p}^{-1}||\leq||F^{-1}||||Q_{4p}^{-1}||||R^{-1}||,
\label{eqcond4}
\end{equation}
hence, multiplying the inequalities \eqref{eqcond3} and \eqref{eqcond4} and considering equations \eqref{eqcond1} and \eqref{eqcond2} we arrive at
$$
\mathrm{Cond}(N_{4p})\leq(2p-1)^2\mathrm{Cond}(Q_{4p}).
$$
By Eq. \eqref{condpartial1} the result follows.
\end{proof}
Next, since now $\mathrm{deg}(R_i^*(x))=i\in\{1,...,p-1\}$, the set $\{R_i^*(\psi)\}_{i=1}^{p-1}$ is not (necessarily) a $\mathbb{Z}$-basis of $\mathcal{O}_{K_{4p}^{+}}$. We overcome this issue with the next result.
\begin{prop}We can write
$$
N_{4p}=PU_{4p},
$$
where
$$
P=\left(
\begin{array}{cccc}
\psi_1 & 0 & \cdots & 0\\
0 & \psi_2 & \cdots & 0\\
\vdots & \vdots & \ddots & \vdots\\
0 & 0 & \cdots & \psi_{p-1}
\end{array}
\right)\mbox{ and  }
U_{4p}=\left(
\begin{array}{cccc}
r_0^*(\psi_1) & r_1^*(\psi_1)  & \cdots & r_{p-2}^*(\psi_1)\\
r_0^*(\psi_2) & r_1*(\psi_2) &  \cdots & r_{p-2}^*(\psi_2)\\
\vdots & \vdots &  \ddots & \vdots\\
r_0^*(\psi_{p-1}) & r_1^*(\psi_{p-1})  & \cdots & r_{p-2}^*(\psi_{p-1})\\
\end{array}
\right).
$$
In addition, for each $i\in\{0,...,p-2\}$, the polynomial $r_i^*(x)\in\mathbb{Z}[x]$ is monic and $\mathrm{deg}(r_i^*(x))=i$. Moreover,
$$
\mathrm{Cond}(U_{4p})\leq p^3(p+1)(2p-1)^2.
$$
\label{ultimaprop}
\end{prop}
\begin{proof}The decomposition $N_{4p}=PU_{4p}$ is clear as it is the fact that for each $i\in\{0,...,p-2\}$, the polynomial $r_i^*(x)\in\mathbb{Z}[x]$ is monic and $\mathrm{deg}(r_i^*(x))=i$. We are left to prove the inequality for the condition number.

First, it is clear that $||P||\leq 2\sqrt{p}$. On the other hand, for each $i\in\{1,...,p-1\}$, we have
$$
|\psi_i|\geq 2\left|cos\left(\frac{\pi}{2}+\frac{\pi}{p}\right)\right|=2\left|sin\left(\frac{\pi}{p}\right)\right|.
$$
Thus, if $p\geq 5$ we have that for each $i\in\{1,...,p-1\}$, it holds $|\psi_i|\geq 2\pi/p$ and thus $||P^{-1}||\leq \frac{p\sqrt{p-1}}{2}$. Hence the result follows
\end{proof}
All told, we can now conclude our main result:
\begin{thm}There exists a matrix $U_{4p}\in M_{p-1}(\mathcal{O}_{K_{4p}^{+}})$,with $\mathrm{Cond}(N_{4p})=O(p^6)$ such that the map
$$
\begin{array}{ccc}
\mathbb{Z}[x]/(\Phi_{4p}^{+}(x)) & \to & \sigma_1(\mathcal{O}_{K_{4p}^{+}})\times ...\times \sigma_{p-1}(\mathcal{O}_{K_{4p}^{+}})\\
\textbf{u} & \mapsto & U_{4p}\textbf{u}
\end{array}
$$
is a lattice isomorphism inducing a polynomial noise increase between the RLWE and the PLWE distributions for $K_{4p}^{+}$. In sum, both problems are equivalent.
\end{thm}
We close our work with numerical illustration (with Matlab) of to what extent our approach drastically reduces the condition number:
\begin{table}[htbp]
\begin{center}
\begin{tabular}{|l|l|l|l|l|}
\hline
Prime & Degree & $\mathrm{Cond}(V_{\Phi_n^{+}})$ &  $\mathrm{Cond}(U_{4p})$ & $4p^6$\\
\hline \hline 
$13$ & $12$ & $1.43\times 10^4$ & $25.92$ & $1.93\times 10^7$\\ \hline
$101$ & $100$ & $1.06\times 10^{19}$ & $583.1$ & $4.24\times 10^{12}$ \\ \hline
$127$ & $126$ & $1.35\times 10^{19}$ & $823.3$ & $1.68\times 10^{13}$\\ \hline
$257$ & $256$ & $6.89\times 10^{23}$ & $2374.05$ & $1.15\times 10^{15}$ \\ \hline
$509$ & $508$ & $4.29\times 10^{27}$ & $18491.2$ & $6.95\times 10^{16}$\\ \hline
\end{tabular}
\label{tablacond}
\end{center}
\end{table}

The column on the right in Table \ref{tablacond} displays the dominant term in the polynomial expression which upper bounds $\mathrm{Cond}(U_{4p})$ according to Prop. \ref{reduction1}. We can see that our bound, even if polynomial, is still rough and actually $\mathrm{Cond}(U_{4p})$ tends to increase much more slowlyy, at least with these figures. For primes beyond 4 digits,  Matlab starts warning about the magnitude of $\mathrm{Cond}(V_{\Phi_n^{+}})$ and that the result might be unreliable. These facts can be seen as an empirical evidence about the improvement granted by our approach.

\subsection{Conclusion, open questions and future work}
For the maximal totally real subextension $K_{4p}^{+}$ of the $4p$-th cyclotomic field ($p$ arbitrary prime), we have showed that the condition number of the Vandermonde matrix $V_{\Phi_{4p}^{+}}$ is exponential in the degree, hence the RLWE and PLWE problems cannot be proved equivalent via the usual isomorphism between $\mathbb{Z}[x]/(\Phi_n^{+}(x))$ and $\sigma\left(\mathcal{O}_{K_{4p}^{+}}\right)$ determined by  $V_{\Phi_{4p}^{+}}$. Instead, we have replaced this isomorphism by another one, attached to a quasi-Vandermonde matrix associated to the Tchebycheff polynomials of degree up to $p-2$ evaluated in the conjugates of the primitive element of $K_{4p}^{+}$ and proved that this matrix is polynomially conditioned in the degree, and consequently, RLWE/PLWE equivalence is proved for this infinite family of number fields.

We outline here three lines of generalisation of our result to $K_n^{+}$ which are under investigation for the monent:
\begin{itemize}
\item From $n=4p$ to $2^rp$ with $r\geq 2$. The nodes are also symmetrically located in this case, hence the classical Vandermonde approach also fails. Going from the Tchebycheff nodes to the Galois conjugates of a primitive element of the totally real subextension requires eliminating the $\frac{jp+1}{2}$-th with $1\leq j\leq 2^{r-2}$. The problem here is to find a uniform upper bound (polynomial in $n$) for the entries of the inverse of the analogue of the matrix $P$ (notations as in Prop. \ref{ultimaprop}).
\item From $n=4p$ to $n=4p_1...p_r$ with $r>1$ and $p_1,...,p_r$ different odd primes. The same problem as before but slightly more complicated since the rows to eliminate are those in position $\frac{p_{i_1}...p_{i_t}+1}{2}$ with $t\leq r$, rather than just involving powers of $2$. In both generalizations, apart from an exercise of alleviating the notations and choice of intermediate lemmas, devising a fine global upper bound for the inverse of the matrix of row operations, seems doable but not trivial.
\item From $n=4p_1...p_r$ to $n=4p_1^{e_1}...p_r^{e_r}$. Unlike the cyclotomic case, where $A(n)=A(rad(n))$, we do not have a similar result for $\Phi_n^{+}(x)$ (notations as in Def. \ref{defncyclo}), and this was one of the main arguments in \cite{blanco} which allowed us to go from the radical conductor case to the general conductor case. This generalization seems much more difficult.
\end{itemize}

Another open question we would like to investigate  is as follows: as recalled in Section 2, Kronecker-Weber's theorem states that every abelian $\mathbb{Q}$-extension is cyclotomic, in the sense that it is a subextension of a certain $K_f$ for minimal integer $f\geq 1$ called the conductor of the extension. Moreover, the determination of all the abelian subextensions of a cyclotomic field $K_n=\mathbb{Q}(\zeta_n)$ can be made fully explicit via the use of Gaussian periods: every such subextension is generated by a primitive element of the form $\sum_{i=1}^{m}\zeta_n^{a_i}$, where the $a_i\in\mathbb{Z}$ determine automorphisms of the Galois group of the extension, a quotient of $\mathbb{Z}_n^*$. The maximal totally real subextension is just the particular case for the Gaussian period $\zeta_n+\zeta_n^{-1}$. The question is, given a cyclotomic subextension in $K_n$, to decide whether or not PLWE is equivalent to RLWE for it.

Finally, another interesting problem is as follows: first, observe that the famillies of polynomials $\{x^i\}_{i\geq 0}$ and $\{T_i(x)\}_{i\geq 0}$ (notations as in Def. \ref{tcheby}) are orthogonal with respect to the Lebesgue measure  (the first in the unit disc, the second in the real line). On the other hand, as we can see in \cite{gautschi}, \cite{pan} or \cite{kuian}, the condition number for a quasi-Vandermonde matrix is linked to the way how the nodes are distributed, with respect to certain measures, either in the unit disc or in a closed and bounded interval. The question in this direction is, given a set of algebraic nodes (Galois conjugated by a given Galois group), to find a family of orthogonal polynomials, compatible with the way they are distributed (either in the real line or unit disc), whose quasi-Vandermonde matrix is polynomially conditioned. Suceeding in this topic, together with the suitable treatment of redundant nodes would allow to produce new families of RLWE/PLWE-equivalent number fields.

\end{document}